\documentclass[a4paper, 12pt]{amsart}
\usepackage{amssymb,verbatim}
\title
[On nonexistence of higher dimensional horizontal Sobolev sets]
{{\bf A low rank property and nonexistence of higher dimensional horizontal Sobolev sets}}
\author{Valentino Magnani}
\address{Valentino Magnani,  Dipartimento di Matematica \\
Largo Pontecorvo 5 \\ I-56127, Pisa}
\email{magnani@dm.unipi.it}
\author{Jan Mal\'y}
\address{Jan Mal\'y, 
Department of Mathematics\\
J. E. Purkyn\v e University\\
\v Cesk\'e ml\'a\-de\v ze~8\\
400~96 \'Ust\'{\i} nad Labem, Czech Republic}
\email{jmaly@physics.ujep.cz}
\author{Samuele Mongodi}
\address{Samuele Mongodi, Scuola Normale Superiore, Piazza dei Cavalieri 7, I-56126 Pisa}
\email{s.mongodi@sns.it}
\thanks{Both the first and third authors acknowledge the support of the European Project ERC AdG *GeMeThNES*.}
\date{\today}

\usepackage{amsmath}
\usepackage{amsfonts}
\usepackage{amssymb}
\usepackage[english]{babel}
\usepackage{color}
\usepackage{bm}
\usepackage{mathrsfs}
\usepackage[latin1]{inputenc}
\usepackage[english]{babel}
\usepackage{amsthm}
\usepackage{graphicx}
\numberwithin{equation}{section}

\newcommand{\cH}{\mathcal{H}}

\newcommand{\cL}{\mathcal{L}}

\newcommand{\cS}{\mathcal{S}}

\newcommand{\cV}{\mathcal{V}}

\renewcommand{\H}{\mathbb{H}}
\newcommand{\intl}{\int\limits}
\newcommand{\R}{\mathbb{R}}
\newcommand{\N}{\mathbb{N}}

\newcommand{\de}{\partial}

\newcommand{\loc}{\mathrm{loc}}
\newcommand{\B}{\mathbb{B}}

\newcommand{\res}{\mbox{\LARGE{$\llcorner$}}}

\newcommand{\der}{\partial}

\newcommand{\e}{\boldsymbol e}
\renewcommand{\u}{\boldsymbol u}
\newcommand{\ve}{\boldsymbol v}
\newcommand{\w}{\boldsymbol w}
\newtheorem{Lmm}{Lemma}[section]
\newtheorem{Prp}[Lmm]{Proposition}
\newtheorem{Teo}{Theorem}[section]

\theoremstyle{definition}
\newtheorem{Def}{Definition}[section]
\newtheorem{Rem}[Lmm]{Remark}

\begin{document}

\begin{abstract}
We establish a ``low rank property'' for Sobolev mappings that pointwise solve a first order nonlinear system of PDEs,
whose smooth solutions have the so-called ``contact property''. As a consequence, Sobolev mappings 
from an open set of the plane, taking values in the first Heisenberg group $\H^1$ and that have almost 
everywhere maximal rank must have images with positive 3-dimensional Hausdorff measure with respect to
the sub-Riemannian distance of $\H^1$.
This provides a complete solution to a question raised in a paper by Z. M. Balogh, R. Hoefer-Isenegger and J. T. Tyson.
Our approach differs from the previous ones. Its technical aspect consists in performing an 
``exterior differentiation by blow-up'', where the standard distributional exterior differentiation is not possible.
This method extends to higher dimensional Sobolev mappings, taking values in higher dimensional Heisenberg groups.
\end{abstract}

\maketitle

\tableofcontents

\pagebreak

\section{Introduction}

It is well known that every noninvolutive tangent distribution on a manifold does not admit any integral submanifold. 
One of the simplest cases is given by the nonintegrable tangent distribution of the first Heisenberg group $\H^1$, 
identified by $\R^3$ with coordinates associated to the left invariant vector fields
\begin{equation}\label{heisv}
 X_1=\de_{x_1}-x_2\de_{x_3}  \quad\mbox{and}\quad X_2=\de_{x_2}+x_1\de_{x_3}\,.
\end{equation}
At each point of the space, these vector fields linearly span a subspace of the tangent space, hence a tangent distribution is defined, corresponding to the so-called ``horizontal subbundle''. 
Although no smooth surfaces in $\H^1$ can be everywhere tangent to $H\H^1$, one may still wonder whether 
there exist more general ``2-dimensional sets'' that can be still considered ``tangent'' to this distribution
in a broad sense. This problem is amazingly related to the study of the Hausdorff dimension of sets with respect to
the sub-Riemannian distance, in short SR-distance, that is associated to $H\H^1$.

In this connection, Z. M. Balogh and J. T. Tyson have constructed an interesting example of
``horizontal fractal'', called the {\em Heisenberg square} $Q_H$, \cite{BalTys05}. 
The 2-dimensional Hausdorff measure of $Q_H$ with respect to both the SR-distance and the Euclidean distance 
is finite and positive, see \cite[Theorem~1.10]{BalTys05}. As proved in \cite{BHIT}, it is possible to find a 
BV function $g:(0,1)^2\to\R$, whose graph $G$ is contained in $Q_H$ and satisfies
\begin{equation}\label{hormet2}
0<\cH_d^2(G)<+\infty.
\end{equation}
The symbol $\cH^2_d$ denotes the Hausdorff measure with respect to the SR-distance $d$ of $\H^1$. 
Condition \eqref{hormet2} never holds for graphs of smooth functions. It can be interpreted 
as a ``metric definition'' of horizontality for lower regular sets. 
In fact, in the general Heisenberg group $\H^n$, represented by $\R^{2n+1}$ equipped by the left invariant vector fields 
\begin{equation}\label{Heisnvf}
 X_i=\de_{x_i}-x_i\de_{x_{2n+1}},\quad X_{n+i}=\de_{x_{n+i}}+x_i\de_{x_{2n+1}}
\quad \mbox{and}\quad i=1,\ldots,n,
\end{equation}
spanning $H\H^n$, every $C^1$ smooth $m$-dimensional submanifold $\Sigma\subset\H^n$
that is everywhere tangent to $H\H^n$ must have the measure $\cH^m_d\res\Sigma$ locally finite.
On the other hand, from Contact Topology, it is well known that 
the nonintegrability of $H\H^n$ is stronger than the noninvolutivity condition of Frobenius Theorem,
since not only hypersurfaces but rather all sufficiently smooth submanifolds $\Sigma\subset\H^n$ of dimension $m$,
with $n<m\leq 2n$, cannot be everywhere tangent to $H\H^n$, in short $T\Sigma\nsubseteq H\H^n$, see for instance \cite[Proposition~1.5.12]{Geiges2008}. 
Thus, when $m>n$ there must exist at least a point $x\in\Sigma$ such that $T_x\Sigma\nsubseteq H_x\H^n$.

This fact has an important metric implication, since the density of $\cH^{m+1}_d\res\Sigma$ 
with respect to the Euclidean surface measure $\cH^m_{|\cdot|}\res\Sigma$ is proportional to the length of the ``vertical tangent $m$-vector'' $\tau_{\Sigma,\cV}$ and this vector vanishes only at those points
$x\in\Sigma$, called {\em horizontal points}, that are characterized by the condition
$T_x\Sigma\subset H_x\H^n$. 

When $\Sigma$ is $C^1$ smooth, the absolute continuity of $\cH^{m+1}_d\res\Sigma$ with respect to $\cH^m_{|\cdot|}$ is mainly a consequence of a higher codimensional negligibility result, \cite{Mag6}, joined with a blow-up at nonhorizontal points, 
\cite{FSSC6, Mag11, Mag12A}. 
The $m$-vector $\tau_{\Sigma,\cV}$ is defined as the projection of the unit tangent $p$-vector of $\Sigma$ onto the orthogonal subspace to the linear space $\Lambda_m(H\H^n)$ of horizontal $m$-vectors, see \cite{Mag11} for more details and related references.
Such results imply that for each smooth $m$-dimensional submanifolds $\Sigma\subset\H^n$ with $m>n$, there holds
\begin{equation}\label{S^m+1}
\cH_d^{m+1}(\Sigma)>0\,.
\end{equation}
In the case $n=1$ and $m=2$, the non-horizontality condition \eqref{S^m+1} for nonsmooth sets 
has been shown in \cite{BHIT}, where $\Sigma$ is a 2-dimensional Lipschitz graph of $\H^1$. 
Here the authors raise the interesting question on the existence of horizontal sets in the sense of \eqref{hormet2} having 
regularity between Lipschitz and BV. 

A first answer to this question is given in \cite{Mag17}, where it is shown that 2-dimensional $W_{\loc}^{1,1}$ 
Sobolev graphs $\Sigma$ in $\H^1$ have to satisfy \eqref{S^m+1}, with $m=2$. 
This approach relies on the fact that for a smooth local parametrization $f:\Omega\to\Sigma$, where $\Omega\subset\R^2$, 
the equation
\begin{equation}\label{eq_horiz}
df^3=f^1df^2-f^2df^1
\end{equation}
only holds at those points $y\in\Omega$ such that $T_{f(y)}\Sigma\subset H\H^1$ and \eqref{eq_horiz}
cannot hold everywhere, since its exterior differentiation would imply that the rank of $Df$ is 
everywhere less than two. 
To see this fact when $f\in W^{1,1}_{\loc}(\Omega,\H^1)$ and it is defined by the graph of a
real-valued Sobolev function, the point is to show that the
almost everywhere validity of \eqref{eq_horiz} allows us to take its distributional exterior differential, 
obtaining that the rank of $Df$ cannot be almost everywhere maximal and this conflicts with the graph structure. 
This is the key to establish \eqref{S^m+1}, since the previous argument shows that \eqref{eq_horiz} fails to hold at least
on a set of positive measure and the Whitney extension theorem yields a $C^1$ smooth submanifold $\tilde\Sigma$
that coincides with the Sobolev graph $\Sigma$ on some measurable subset $A\subset\tilde\Sigma\cap\Sigma$
of positive Euclidean surface measure, where in addition $TA\nsubseteq H\H^n$. 
As a consequence, in view of the previous comments on the density of $\cH^3_d\res\tilde\Sigma$, we achieve
\[
\cH^3_d(\Sigma)\geq\cH^3_d(A)>0\,.
\]
More generally, the same argument applies to all cases where we are able to show that \eqref{eq_horiz} cannot hold almost everywhere. To show this fact in other cases of low regular sets, we need the summability of both $f$ and $Df$ to allow for the distributional exterior differentiation of \eqref{eq_horiz}. The distributional exterior differential of $f^1df^2-f^2df^1$ is
exactly twice the distributional Jacobian of the mapping $(f^1,f^2)$, hence assuming for instance 
that $(f^1,f^2)\in W_{\loc}^{1,p}(\Omega,\R^2)$ with $p\geq 4/3$, we obtain that this distributional Jacobian is well defined.
As a consequence, every image $\Sigma$ of a mapping in $W^{1,p}_{loc}(\Omega,\R^3$ with
$p\geq 4/3$ and whose Jacobian matrix has almost everywhere maximal rank must satisfy \eqref{S^m+1} with $m=2$, \cite{Mag17}. The validity of this result in the case $1\leq p< 4/3$ was left open, since the distributional Jacobian cannot be defined. The following theorem answers this question.
\begin{Teo}\label{m=2}
Let $\Omega\subset\R^2$ be open, let $f\in W_{\loc}^{1,1}(\Omega,\R^3)$ be such that
the Jacobian matrix $Df$ has almost everywhere maximal rank and define $\Sigma=f(\Omega)$. 
It follows that $\cH_d^3(\Sigma)>0$.
\end{Teo}
This completes the answer to the previously mentioned question raised in \cite{BHIT}.
Our approach differs from the previous ones and it can be applied to every Heisenberg group $\H^n$,
that we identify with $\R^{2n+1}$ as a linear space.
We consider $f:\Omega\to\R^{2n+1}$, where $\Omega$ is an open set of $\R^m$. In this case, the 
horizontality condition for $f$ is given by the equation
\begin{equation}\label{eq_horizn}
df^{2n+1}=\sum_{j=1}^n \big(f^jdf^{j+n}-f^{j+n}df^j\big)\,.
\end{equation}
The previous arguments apply if we are able to show that the almost everywhere validity of 
\eqref{eq_horizn} implies a {\em low rank property}, namely,
$Df$ must have rank less than $n+1$
almost everywhere in $\Omega$. Clearly, we will apply such a result in the nontrivial case $n+1\leq m\leq 2n$. 
We will assume that $f\in  W^{1,1}_{\loc}(\Omega,\R^{2n+1})$.
Let us summarize the main idea of the proof. First, assume that $m=2$.
We perform a kind of ``exterior differentiation by blow-up'', rescaling $f$ at Lebesgue points $z\in\Omega$
of both $f$ and $D f$. The rescaled functions $f_{z,\rho}$, introduced in Definition~\ref{rescaled}, are defined 
on the unit ball $\B$ of $\R^2$ for all $\rho>0$ sufficiently small and converge
to the linear mapping
$u:y\mapsto Df(z)\cdot y$ in $W^{1,1}(\B)$ as $\rho\to 0_+$.
The almost everywhere pointwise validity of \eqref{eq_horizn} implies that the one-form
\begin{equation}\label{exactform}
\sum_{j=1}^n \Big(f^j_{z,\rho}\,df_{z,\rho}^{j+n}-f^{j+n}_{z,\rho}\,df_{z,\rho}^j\Big)
\end{equation}
is ``weakly exact'' in the sense that it is a.e.\ equal to $dw_\rho$ for some $w_\rho\in W^{1,1}(\B)$, see Lemma~\ref{lmm_exactness}.
We exploit this fact by integrating \eqref{exactform} on the Euclidean sphere $\de B(0,r)$ for almost every $r\in(0,1)$
and pass to the limit with  respect to $\rho$ as it goes to zero by a suitable positive infinitesimal sequence $(\rho_k)$. 
%Only at this step, we need the condition \eqref{which_p}
%  -- I think we need not to emphasize so much.
%
%
Since the blow-up limit has the form 
$$
\sum_{j=1}^n \Big(u^j\,du^{j+n}-u^{j+n}\,du^j\Big)
$$
with $u(y)=Df(z)\cdot y$, we obtain that its oriented integral on almost every sphere vanishes, 
hence the Stokes theorem implies that 
\begin{equation}\label{df_j}
\sum_{j=1}^ndf^j(z)\wedge df^{j+n}(z)=0\,.
\end{equation}
Now, if $m>2$, we obtain \eqref{df_j} by a slicing argument, so that
the whole range $m\ge 2$ is provided.
We will deduce from \eqref{df_j} that the 
rank of $Df(z)$ is less than $n+1$, 
so this rank condition holds almost everywhere, eventually leading us to our Theorem~\ref{teo_lowrank}.
According to this theorem, Sobolev mappings that satisfy the horizontality condition \eqref{eq_horizn} 
almost everywhere must satisfy a ``low rank property''.
This fact should be seen somehow as a ``differential obstruction''. 
It is worth to compare this obstruction with the ``Lipschitz obstructions'' appearing in the study of
Lipschitz homotopy groups of the Heinsenberg group, \cite{DJHLT}.
The main application of Theorem~\ref{teo_lowrank} is the following result.

\begin{Teo}\label{H_dm}
Let $\Omega\subset\R^m$ be an open set, let $n<m\leq 2n$ and let $f\in W_{\loc}^{1,1}(\Omega,\R^{2n+1})$.
Suppose that the Jacobian matrix $Df$ has rank equal to $m$ almost everywhere and set $\Sigma=f(\Omega)$. 
Then $\cH^{m+1}_d(\Sigma)>0$.
\end{Teo}
We remark that in the case $m=2$ and $n=1$, this theorem exactly yields Theorem~\ref{m=2}.
In ending, we wish to point out a curious observation on the graph $G$ of the BV function $g$
mentioned above, since we can translate the metric horizontality of \eqref{hormet2}
into a somehow  ``tangential condition''. In fact, as a byproduct
of our techniques, one can easily observe that the approximate differential of the graph mapping $f=(x_1,x_2,g)$ must satisfy \eqref{eq_horiz} almost everywhere, hence $\mbox{ap}\,\nabla g=(-x_2,x_1)$ almost everywhere, see Theorem~\ref{BVg}. 
This can be seen as a tangential condition in the sense of Geometric Measure Theory.

%%%%%%%%%%%%%%%%%%%%%%%%%%%%%%%%%%%%%%%%%%%%%%%%%%%%%%%%%%%%%%%%%%%%%%%%%%%%%%%%%%%%%%%%%%%%%%%%%%%%%%%%%%%%%%%%%%%%%%%
%
%
%
%
\section{Slicing}
\label{s:sl}
%
%
%
%
%%%%%%%%%%%%%%%%%%%%%%%%%%%%%%%%%%%%%%%%%%%%%%%%%%%%%%%%%%%%%%%%%%%%%%%%%%%%%%%%%%%%%%%%%%%%%%%%%%%%%%%%%%%%%%%%%%%%%%%

For the reader's convenience, in this section we recall some well known facts 
about Sobolev sections, that will be used in the subsequent part of the paper.
Let $m$ be a positive integer and denote by $(\e_1,\dots,\e_m)$ the canonical basis of $\R^m$.
If $\Gamma\subset\{1,\dots,m\}$ is a set of indices, then $V_{\Gamma}$ is 
the linear span of $\{\e_j\colon j\in\Gamma\}$
and $V_{\Gamma}^{\bot}$ is the linear span of $\{\e_j\colon j\in\{1,\dots,m\}\setminus\Gamma\}$.
We introduce the orthogonal projections
$$
\pi_{\Gamma}(x)=\sum_{j\in\Gamma}x_j\e_j\quad\mbox{and}\quad 
\hat\pi_{\Gamma}(x)=x-\pi_{\Gamma}(x) 
$$
where $x\in\R^m$, $\pi_{\Gamma}:\; \R^m\to V_{\Gamma}$ and $\hat\pi_{\Gamma}:\;\R^m\to V_{\Gamma}^{\bot}$.
Let $Q$ be an open $m$-dimensional interval in $\R^m$, namely the product of $m$ open intervals,
and fix a nonempty subset $\Gamma\subsetneq\{1,\dots,m\}$.
We define the  projected intervals
$$
Q_{\Gamma}=\pi_{\Gamma}(Q) \quad\mbox{and}\quad \hat Q_{\Gamma}=\hat\pi_{\Gamma}(Q).
$$
If $u\colon Q\to\R$ is a function and  $z\in \hat Q_{\Gamma}$, we define the \textit{section} $u^z:Q_{\Gamma}\to\R$ as
$$
u^z(y)=u(z+y),\qquad y\in Q_{\Gamma}.
$$
\begin{Def}\rm
We say that a sequence $\{u_h\}$ in a Banach space $(X,\|\cdot\|)$  converges \textit{fast} to $u\in X$,
or that it is {\em fast convergent}, if $\sum_{h=1}^{\infty}\|u_h-u\|<\infty$.
\end{Def}
We wish to point out that the fast convergence in $W^{1,1}$ is just 
the joint fast convergence in $L^1$ of functions and their gradients.
As a consequence of both Fubini's theorem and Beppo Levi's convergence theorem for series, we get the next proposition. 
\begin{Prp}\label{p:fast} 
Let $\{u_h\}\subset W^{1,1}(Q)$ be a sequence which converges 
fast to $u\in W^{1,1}(Q)$.
Then for each $k=1,\dots,m$ and for almost every $z\in\hat Q_{\Gamma}$ we have 
$u^z,(\partial_{y_k}u)^z,u_h^z,(\partial_{y_k}u_h)^z\in L^{1}(Q_\Gamma)$, $h\in\N$, further, 
$\{u_h^z\}$ converges fast to $u^z$ in $L^{1}(Q_{\Gamma})$ and
$(\partial_{y_k}u_h)^z$ converges fast to $(\partial_{y_k}u)^z$ 
in $L^{1}(Q_{\Gamma})$.
\end{Prp}
Each $u\in W^{1,1}(Q)$ is a limit of a fast convergent sequence of smooth functions.
Applying Proposition \ref{p:fast} we obtain the following consequence.
\begin{Prp}\label{p:der} 
Let $u\in W^{1,1}(Q)$. Then for almost every $z\in\hat Q_{\Gamma}$ we have 
$u^z\in W^{1,1}(Q_{\Gamma})$ and
\begin{equation}\label{pders}
\der_{y_k}u^z=(\der_{y_k}u)^z\quad \text{a.e.\ in }Q_{\Gamma},\qquad k=1,\dots,m.
\end{equation}
\end{Prp}
Summarizing Propositions \ref{p:fast} and \ref{p:der} we obtain the following.

\begin{Prp}\label{p:fast1} 
Let $\{u_h\}\subset W^{1,1}(Q)$ be a sequence which converges 
fast to $u\in W^{1,1}(Q)$.
Then for almost every $z\in\hat Q_{\Gamma}$ we have 
$u^z,u_h^z\in W^{1,1}(Q_\Gamma)$, $h\in\N$, and
$\{u_h^z\}$ converges fast to $u^z$ in $W^{1,1}(Q_{\Gamma})$.
\end{Prp}

%
%
%
%
%
%
%%%%%%%%%%%%%%%%%%%%%%%%%%%%%%%%%%%%%%%%%%%%%%%%%%%%%%%%%%%%%%%%%%%%%%%%%%%%%%%%%%%%%%%%%%%%%%%%%%%%%%%%%%%%%%%%%%%%%%
%
%
%
%
\section{Oriented integration on the circle}
%
%
%
%
%%%%%%%%%%%%%%%%%%%%%%%%%%%%%%%%%%%%%%%%%%%%%%%%%%%%%%%%%%%%%%%%%%%%%%%%%%%%%%%%%%%%%%%%%%%%%%%%%%%%%%%%%%%%%%%%%%%%%%%
%
%
%
%
%
%
The idea of slicing can be also applied to behavior of Sobolev functions
on a.e.\ sphere. However, for our purposes it is enough to 
perform this analysis in $\R^2$ only, so that we will study
Sobolev spaces on circles. 
The open ball in $\R^2$ with center at $x$ and radius $r$ is 
denoted by $B(x,r)$.
\begin{Def}[Function spaces on the circle]\rm
Consider the circle $\partial B(x,r)$ and its parametrization
\begin{equation}\label{polar}
\psi(t)=(x_1+r\cos t,\,x_2+r\sin t),\qquad t\in\R.
\end{equation}
We define $\psi_-=\psi\lfloor_{(-\pi,\pi)}$ and $\psi_+=\psi\lfloor_{(0,2\pi)}$, hence 
$(\psi_+,\psi_-)$ is an oriented 
atlas of $\partial B(x,r)$. This atlas automatically defines function spaces on $\der B(x,r)$.
Let $X$ be a generic function space symbol which may refer e.g.\ to
$W^{1,p}$, $L^p$ or $C$.
We say that $u:\partial B(x,r)\to \R$ belongs to $X(\der B(x,r))$
if $u\circ\psi_-$ 
belongs to  $X((-\pi,\pi)$ and $u\circ\psi_+$ belongs to $X(0,2\pi))$.
\end{Def}
\begin{Def}[Integrable forms on the circle]\label{OrientedI}\rm 
Let us consider $u,v:\der B(x,r)\to\R$. Then the
{\em oriented integral} of the differential form $u\,dv$ is defined as follows
$$
\int_{\de B(x,r)} u\,dv=\int_{-\pi}^{\pi}
(u\circ \psi)(t)\,(v\circ \psi)'(t)\,dt,
$$
whenever this expression has a good sense, if e.g.\
$u\in L^\infty(\der B(x,r))$, $v\in W^{1,1}(\der B(x,r))$ and
$(v\circ\psi)'$ is the distributional derivative of $v\circ\psi$.
\end{Def}
The following lemma relates the fast convergence with the convergence of 
oriented integrals on spherical sections.
\begin{Lmm}\label{lmm_passage2}
Let $u,u_h,v,v_h\in W^{1,1}(B(x,\rho))$, $h\in\N$, and suppose that 
both $u_h\to u$ and $v_h\to v$  fast in $W^{1,1}(B(x,\rho))$.
Then for almost every $0<r<\rho$ 
the restrictions of $u,u_h,v,v_h$ to $\partial B(x,r)$ belong to
$W^{1,1}(\der B(x,r))$ and 
\begin{equation}\label{formconv}
\intl_{\de B(x,r)} u_h\,dv_h\to \intl_{\de B(x,r)}u\,dv\,.
\end{equation}
\end{Lmm}

\begin{proof} 
We use the polar coordinates given by $\Psi(r,t)=(x_1+r\cos t,\;x_2+r\sin t)$
and the notation $\Psi^r=\Psi(r,\cdot)$, $r\in (0,\rho)$.
First, we observe that given $w\in W^{1,1}(B(x,\rho))$,
then $w\circ\Psi$
belongs to $W^{1,1}((\delta,\rho)\times (-2\pi,2\pi))$ for each 
$\delta\in (0,\rho)$. 
The fast convergence of both $\{u_h\}$ and $\{v_h\}$ in $W^{1,1}(B(x,r))$ implies that
$u_h\circ\Psi$ and $v_h\circ\Psi$ are fast convergent in $W^{1,1}\big((\delta,\rho)\times(-2\pi,2\pi)\big)$
with limits equal to $u\circ\Psi$ and $v\circ\Psi$, respectively.
By Proposition~\ref{p:fast1}, for a.e. $r\in(\delta,\rho)$ we have that 
$u_h\circ\Psi^r,v_h\circ\Psi^r,u\circ\Psi^r,v\circ\Psi^r\in W^{1,1}\big((-2\pi,2\pi)\big)$
and both $u_h\circ\Psi^r$ and $v_h\circ\Psi^r$ converge fast in $W^{1,1}\big((-2\pi,2\pi)\big)$
to $u\circ\Psi^r$ and $v\circ\Psi^r$, respectively. 

Fix such a good radius $r$. Then $u\circ\Psi^r$, $u_h\circ\Psi^r$ are absolutely
continuous up to a modification on a null set. Using the one-dimensional Sobolev
embedding and
passing to absolutely continuous representatives, we obtain a uniform convergence $u_h\circ\Psi^r\to u\circ\Psi^r$.
Joining with the $L^1$-convergence $(v_h\circ\Psi^r)'\to (v\circ\Psi^r)'$ we
conclude that
\[
\aligned
\int_{\der B(x,r)} u_h\ dv_h&=\int_{-\pi}^\pi (u_h\circ\Psi^r)(t) (v_h\circ\Psi^r)'(t)\,dt\to 
\int_{-\pi}^\pi (u\circ\Psi^r)(t) (v\circ\Psi^r)'(t)\,dt
\\&=\int_{\der B(x,r)} u\ dv
\endaligned
\]
as required.
By the arbitrary choice of $\delta>0$, we have proved that \eqref{formconv} holds for a.e. $r\in(0,\rho)$. 
\end{proof}
\begin{Lmm}\label{l_stokes}
Let  $v\in W^{1,1}(B(x,\rho))$. For almost every $r\in(0,\rho)$,
the oriented integral $ \displaystyle \intl_{\der B(x,r)}dv $ is well defined and equal to zero.
\end{Lmm}

\begin{proof} Again, we use the polar coordinates as in the preceding proof.
By Proposition \ref{p:der}, for a.e.\ $r\in (0,\rho)$, the section
$v\circ\Psi^r$ belongs to $W^{1,1}(-2\pi,2\pi)$. If $\bar v\circ\Psi^r$ is the absolutely continuous representative of $v\circ\Psi^r$, we have
$$
\int_{\der B(x,r)}dv=\int_{-\pi}^{\pi}(v\circ\Psi^r)'(t)\,dt
=\bar v\circ\Psi^r(\pi)-\bar v\circ\Psi^r(-\pi)=0,
$$
as $\bar v\circ\Psi^r$ is obviously $2\pi$-periodic.
\end{proof}

%
%
%
%
%%%%%%%%%%%%%%%%%%%%%%%%%%%%%%%%%%%%%%%%%%%%%%%%%%%%%%%%%%%%%%%%%%%%%%%%%%%%%%%%%%%%%%%%%%%%%%%%%%%%%%%%%%%%%%%%%%%%%%
%
%
%
%
\section{An exterior differentiation by blow up}
%
%
%
%
%%%%%%%%%%%%%%%%%%%%%%%%%%%%%%%%%%%%%%%%%%%%%%%%%%%%%%%%%%%%%%%%%%%%%%%%%%%%%%%%%%%%%%%%%%%%%%%%%%%%%%%%%%%%%%%%%%%%%%%
%
%
%
%
%

Throughout this section, we fix an open set $\Omega\subset\R^2$, a mapping 
$f\in W_{\loc}^{1,1}(\Omega,\R^{2n+1})$ and a point $z\in\Omega$ that is a Lebesgue point of both 
$f$ and $Df$. Recall that $z$ is a Lebesgue point for a measurable function $u$ if 
$$
\lim_{r\to 0_+}r^{-n}\int_{B(z,r)}|u(y)-u(z)|\,dy=0
$$
and that almost every point is a Lebesgue point of $u$ if $u\in L^1_{loc}(\Omega)$.
As already pointed out in the introduction, $\H^n$ is identified with $\R^{2n+1}$ 
equipped with the vector fields of \eqref{Heisnvf}. We fix $\rho>0$ such that $\overline{B(z,\rho)}\subset\Omega$. 
Finally, the open unit ball in $\R^2$ centered at the origin will be denoted by $\B$.
\begin{Def}\label{rescaled}\rm
Let $0<r<\rho$ and define the {\em rescaled function} $f_{z,r}:\B\to\R^{2n+1}$ as 
$$
f_{z,r}(y):=\frac{f(z+ry)-f(z)}{r}\;.
$$
\end{Def}
Obviously, $f_{z,r}\in W^{1,1}(\B, \R^{2n+1})$ is well defined whenever $0<r\leq \rho$.
We use the assumption that $z$ is a Lebesgue point of both $f$ and $Df$ to conclude that
\begin{equation}\label{blowup}
\lim_{r\to 0_+}\int_{\B}|f_{z,r}(y)-Df(z)\cdot y|\,dy=0,
\end{equation}
cf.\ e.g.\ \cite[Theorem 3.4.2]{Zie}.
The next lemma provides us with important information on the rescaled function $f_{z,\rho}$.
\begin{Lmm}\label{lmm_exactness}
If \eqref{eq_horizn} holds almost everywhere, then there exists $w\in W^{1,1}(\B)$ such that 
$$
dw(y)=\sum_{j=1}^n 
f^j_{z,\rho}(y)df^{j+n}_{z,\rho}(y)-f^{j+n}_{z,\rho}(y)\,
df^j_{z,\rho}(y)\quad\mbox{for a.e.}\quad y\in\B\,.
$$
\end{Lmm}
\begin{proof} In view of \eqref{eq_horizn}, it follows that 
$$
\nabla f^{2n+1}_{z,\rho}(y)=\nabla f^{2n+1}(z+\rho y)=\sum_{j=1}^nf^j(z+\rho y)\nabla f^{j+n}(z+\rho y)-f^{j+n}(z+\rho y)\nabla f^{j}(z+\rho y)
$$
for a.e. $y\in\B$. We add and subtract all terms of the form $f^j(z)\nabla f^{j+n}(z+\rho y)$, getting
\begin{equation*}\begin{split}\nabla f^{2n+1}_{z,\rho}&(y)=\sum_{j=1}^nf^j(z+\rho y)\nabla f^{j+n}(z+\rho y)-f^{j+n}(z+\rho y)\nabla f^{j}(z+\rho y)\\
=&\sum_{j=1}^n\left(f^j(z+\rho y)-f^j(z)\right)\nabla f^{j+n}(z+\rho y)-\left(f^{j+n}(z+\rho y)-f^{j+n}(z)\right)\nabla f^{j}(z+\rho y)\\
+&\sum_{j=1}^n f^j(z)\nabla f^{j+n}(z+\rho y)-f^{j+n}(z)\nabla f^{j}(z+\rho y)\;.\end{split}\end{equation*}
Dividing by $\rho$, we can rewrite the previous equation as follows
$$\frac{1}{\rho}\left\{\nabla f^{2n+1}_{z,\rho}(y)-\sum_{j=1}^n\left(f^j(z)\nabla f^{j+n}(z+\rho y)-f^{j+n}(z)\nabla f^{j}(z+\rho y)\right)\right\}$$
$$=\sum_{j=1}^n f^j_{z,\rho}(y)\nabla f^{j+n}(z+\rho y)-f^{j+n}_{z,\rho}(y)\nabla f^{j}(z+\rho y)\;.$$
Since $\nabla f(z+\rho y)=\nabla f_{z,\rho}(y)$, 
this immediately leads to the conclusion.
\end{proof}

\medskip
Next, we show that, under sufficient integrability conditions, it is possible to take somehow 
the differential of both sides of (\ref{eq_horizn}), achieving the following theorem.
\begin{Lmm}\label{prp_2form2}
If \eqref{eq_horizn} holds almost everywhere, then we have 
\[\displaystyle \sum_{j=1}^ndf^j(z)\wedge df^{j+n}(z)=0\;.\]
\end{Lmm} 
\begin{proof}
We choose $\rho_h\searrow0$ such that $\rho_1<\rho$ and set
$u_h=f_{z,\rho_h}$.
By Lemma~\ref{lmm_exactness}, 
there exists $w_h\in W^{1,1}(\B)$ such that for $\cL^2$-almost every $y\in\B$ we have
$$
dw_h(y)=\sum_{j=1}^n u_h^j(y)\,du_h^{j+n}(y)-u_h^{j+n}(y)\,du_h^j(y)\;.
$$
Furthermore, since $z$ is a Lebesgue point of both $f$ and $Df$, it follows that
\begin{equation}\label{SobolevConv}
u_h\to u\;\text{ in }\; W^{1,1}(\B),\quad \text{where}\quad
u(y)=\nabla f(z)\cdot y,\quad y\in\B.
\end{equation}
We may assume that the sequence $\rho_h$ is defined in such a way that the
convergence in \eqref{SobolevConv} is fast.
\relax Lemma~\ref{lmm_passage2} implies that for almost every $r\in(0,1)$ the integral
$$
\int_{\partial B(0,r)}\Bigl(\sum_{j=1}^n u_h^j\,du_h^{j+n}-u_h^{j+n}\,du_h^{j}\Bigr)
$$
is well defined and equal to $\int_{\partial B(0,r)}dw_h$. 
Thus, in view of Lemma~\ref{l_stokes} we have 
$$
\int_{\partial B(0,r)} \Bigl(\sum_{j=1}^n u_h^j\,du_h^{j+n}-u_h^{j+n}\,du_h^{j}\Bigr)
=\int_{\partial B(0,r)}dw_h=0
$$
for all $h$ and almost every $r\in (0,1)$.
Taking into account both \eqref{SobolevConv} and Lemma~\ref{lmm_passage2}, 
for almost every $r\in (0,1)$ we have
$$
0=\int_{\partial B(0,r)} \Bigl(\sum_{j=1}^n u_h^j\,du_h^{j+n}-u_h^{j+n}\,du_h^{j}\Bigr)\to\int_{\partial B(0,r)}
\Bigl(\sum_{j=1}^n u^j\,du^{j+n}-u^{j+n}\,du^{j}\Bigr).
$$
It is enough to pick one such a radius, so that by Stokes theorem, we obtain
\begin{equation}\label{ujm=2}
\int_{ B(0,r)} \sum_{j=1}^n du^j\wedge du^{j+n}=0.
\end{equation}
The equation \eqref{ujm=2} yields
$$
\cL^2(B(0,r))\,\sum_{j=1}^n\det\bigl(\nabla f^j(z),\nabla f^{j+n}(z)\bigr)=0\,.
$$
Thus, we have 
$\sum_{j=1}^n\det\bigl(\nabla f^j(z),\nabla f^{j+n}(z)\bigr)=0$
which gives our claim.
\end{proof}

%%%%%%%%%%%%%%%%%%%%%%%%%%%%%%%%%%%%%%%%%%%%%%%%%%%%%%%%%%%%%%%%%%%%%%%%%%%%%%%%%%%%%%%%%%%%%%%%%%%%%%%%%%%%%%%%%%%%%%%
%
%
%
%
\section{The $m$-dimensional case}
%
%
%
%
%%%%%%%%%%%%%%%%%%%%%%%%%%%%%%%%%%%%%%%%%%%%%%%%%%%%%%%%%%%%%%%%%%%%%%%%%%%%%%%%%%%%%%%%%%%%%%%%%%%%%%%%%%%%%%%%%%%%%%%

In this section we treat the general case $m\ge2$.

\begin{Teo}\label{prp_2form} Let $\Omega\subset\R^m$ be open, let $m\geq2$ and set and $f\in W_{\loc}^{1,1}(\Omega,\R^{2n+1})$.
If \eqref{eq_horizn} holds almost everywhere, then almost everywhere we have 
\begin{equation}\label{final}
 \sum_{j=1}^ndf^j\wedge df^{j+n}=0\,.
\end{equation}
\end{Teo} 
\begin{proof} It is enough to verify \eqref{final} on an arbitrary
$m$-dimensional open cube $Q\subset\subset\Omega$.
Fix $1\leq k<l\leq m$. 
We set $\Gamma=\{k,l\}$ and use the notation of Section \ref{s:sl}, 
with the exception that now we use the \textit{subscript} $z$ to denote the section
$$
f_z(y)=f(z+y),\qquad y\in Q_{\Gamma}.
$$
By Proposition~\ref{p:der}, for a.e.\ $z\in \hat Q_{\Gamma}$ we have
that $f_z\in W^{1,1}(Q_{\Gamma})$ and
\begin{equation}\label{pders2}
\frac{\partial f_z}{\partial x_k}=\Bigl(\frac{\partial f}{\partial x_k}\Bigr)_z\;,
\qquad
\frac{\partial f_z}{\partial x_l}=\Bigl(\frac{\partial f}{\partial x_l}\Bigr)_z
\quad\text{a.e. in }Q_{\Gamma}.
\end{equation}
In particular, we have
$$
df_z^{2n+1}=\sum_{j=1}^n \big(f_z^jdf_z^{j+n}-f_z^{j+n}df_z^j\big)
\quad\text{a.e. in }Q_{\Gamma}.
$$
Then use Lemma \ref{prp_2form2} on $Q_{\Gamma}$ to infer that
$$
\sum_{j=1}^ndf_z^j\wedge df_z^{j+n}=0 \quad\text{a.e. in }Q_{\Gamma}.
$$
Using Fubini's theorem and \eqref{pders2} we obtain that 
$$
\sum_{j=1}^n\det
\begin{pmatrix}
\frac{\partial f^j}{\partial x_k}, &\frac{\partial f^{j}}{\partial x_l}\\
\frac{\partial f^{j+n}}{\partial x_k}, &\frac{\partial f^{j+n}}{\partial x_l}
\end{pmatrix}
=0
 \quad\text{a.e. in }Q.
$$
By the arbitrary choice of $k$ and $l$, the equality \eqref{final} holds a.e.\ in $Q$.
\end{proof}

%%%%%%%%%%%%%%%%%%%%%%%%%%%%%%%%%%%%%%%%%%%%%%%%%%%%%%%%%%%%%%%%%%%%%%%%%%%%%%%%%%%%%%%%%%%%%%%%%%%%%%%%%%%%%%%%%%%%%%%
%
%
%
%

%%%%%%%%%%%%%%%%%%%%%%%%%%%%%%%%%%%%%%%%%%%%%%%%%%%%%%%%%%%%%%%%%%%%%%%%%%%%%%%%%%%%%%%%%%%%%%%%%%%%%%%%%%%%%%%%%%%%%%%
%
%
%
%
\section{Non-horizontality of higher dimensional Sobolev sets}
%
%
%
%
%%%%%%%%%%%%%%%%%%%%%%%%%%%%%%%%%%%%%%%%%%%%%%%%%%%%%%%%%%%%%%%%%%%%%%%%%%%%%%%%%%%%%%%%%%%%%%%%%%%%%%%%%%%%%%%%%%%%%%%

In this section, the positive integers $m$ and $n$ will be assumed to satisfy the condition $n+1\leq m\leq 2n$.

\begin{Lmm}\label{lmm_lowrank} 
Let $\u_1,\dots,\u_{2n}\in\R^m$.
Assume that 
$$
\sum_{j=1}^n \u_j\wedge \u_{j+n}=0.
$$
Then the matrix $B$ with rows $\u_1,\dots,\u_{2n}$ has rank at most $n$.
\end{Lmm}

\begin{proof}
%Let $P$ be the matrix with
%rows $\u_1,\dots,\u_n$ and $Q$ be the matrix with
%rows $\u_{n+1},\dots,\u_{2n}$, so that $B=\binom{P}{Q}$. 
We denote the inner product in $\R^{2n}$ by $\langle\cdot,\cdot\rangle$. Further,
$(\e_1,\dots,\e_{2n})$ is the canonical basis of 
$\R^{2n}$ and $I_n$ is the $n\times n$ identity matrix.
We consider the $2n\times 2n$ matrix
$$
J=\begin{pmatrix}0&-I_n\\I_n&0\end{pmatrix}.
$$
Choose $\ve=(v_1,\dots,v_m)$, 
$\w=(w_1,\dots,w_m)\in\R^m$. We have
$$
B\w=\sum_{i=1}^{n}\sum_{k=1}^m
(b_i^{k}w_{k}\e_i+b_{i+n}^{k}w_{k}\e_{i+n})\ ,\qquad
JB\ve=\sum_{j=1}^{n}\sum_{l=1}^m
(b_{j}^{l}v_{l}\e_{j+n}-b_{j+n}^{l}v_{l}\e_j)
$$
and this implies that
$$
\langle B\w,\;JB\ve\rangle
=\sum_{k,l=1}^m
\sum_{i,j=1}^n
\langle b_i^{k}w_{k}\e_i+b_{i+n}^{k}w_{k}\e_{i+n},\;
b_{j}^{l}v_{l}\e_{j+n}-b_{j+n}^{l}v_{l}\e_j\rangle.
$$
The summands are nonzero only for $i=j$, in which case
$$
\langle b_i^{k}w_{k}\e_i+b_{i+n}^{k}w_{k}\e_{i+n},\;
b_{i}^{l}v_{l}\e_{i+n}-b_{i+n}^{l}v_{l}\e_i\rangle=
w_{k}v_{l}\det
\begin{pmatrix}b_i^{l},&b_i^{k}\\
b_{i+n}^{l},&b_{i+n}^{k}
\end{pmatrix}\;,
$$
so that
$$
\langle B\w,\;JB\ve\rangle
=\sum_{k,l=1}^mw_{k}v_{l}\sum_{i=1}^n
\det
\begin{pmatrix}b_i^{l},&b_i^{k}\\
b_{i+n}^{l},&b_{i+n}^{k}\end{pmatrix}
=\sum_{k,l=1}^mw_{k}v_{l}
\Bigl(\sum_{i=1}^n\u_i\wedge \u_{i+n}\Bigr)_{l,k}
=0.
$$
Then the images of $B$ and of $JB$ are orthogonal subspaces of $\R^{2n}$, having the same dimension,
hence the rank of $B$ cannot be greater than $n$.
\end{proof}
\begin{Teo}\label{teo_lowrank}  Let $\Omega\subseteq\R^m$ be an open set and consider 
$f\in W^{1,1}_{\loc}(\Omega,\R^{2n+1})$ which almost everywhere satisfies (\ref{eq_horizn}). 
It follows that the Jacobian matrix of $f$ almost everywhere has rank at most $n$.
\end{Teo}
\begin{proof}
This is a consequence of Theorem~\ref{prp_2form} and Lemma \ref{lmm_lowrank}.
\end{proof}

\medskip

By Theorem~\ref{teo_lowrank}, the proof of Theorem~\ref{H_dm} follows essentially the same lines
of \cite{Mag17}. Next, for the sake of the reader, we adapted this proof to our setting.

\medskip 
\begin{proof}[Proof of Theorem~\ref{H_dm}]
By Theorem~\ref{teo_lowrank}, the equation (\ref{eq_horizn}) fails to hold for $f$ on a set $E\subset\Omega$ of positive $\cL^m$-measure. We can assume that $E$ is bounded, made by density points, that everywhere on $E$ the approximate differentiable of $f$ exists and equals its distributional differential and they have everywhere rank equal to $m$. 
Up to taking a smaller piece of $E$, we can also assume that $f$ is Lipschitz. 
Then we consider a Lipschitz extension of
$f\vert_E$ to all of $\R^m$ and apply Whitney extension theorem, hence finding a subset $E_0$ of $E$ with positive measure and $g\in C^1(\R^m,\R^{2n+1})$ such that $g\vert_{E_0}=f\vert_{E_0}$ and the approximate differential of $f$ and the differential of $g$ coincide on $E_0$.
We choose $y_0\in E_0$ and notice that for a fixed $r_0>0$ sufficiently small, we have 
$\cL^m(B_{y_0,r_0}\cap E_0)>0$ and $\Sigma_0=g(B_{y_0,r_0})$ is an $m$-dimensional
embedded manifold of $\R^{2n+1}$.
By the properties of $g$ and the classical area formula, we have 
\[
\Sigma_1=f(B_{y_0,r_0}\cap E_0)=g(B_{y_0,r_0}\cap E_0)\subset\Sigma_0\cap\Sigma\quad\mbox{and}\quad \cH^m_{|\cdot|}(\Sigma_1)>0.
\]
Since (\ref{eq_horizn}) does not hold on $E_0$, for any $y\in B_{y_0,r_0}\cap E_0$, we have 
$T_{f(y)}\Sigma_0\not\subset H_y\H^n$, therefore
\[
\tau_{\Sigma_0, \cV}(f(y))\neq0,
\]
where we have used the notation $\tau_{\Sigma_0, \cV}(x)$ with $x\in\Sigma_0$ to indicate
the {\em vertical tangent $p$-vector} to $\Sigma_0$ at $x$, see \cite[Definition~2.14]{Mag11}.
This $m$-vector vanishes exactly at those points $x$ where $T_x\Sigma_0\subset H_x\H^n$,
see \cite[Proposition~3.1]{Mag11}. From both \cite{Mag6} and \cite{Mag11}, the spherical Hausdorff measure $\cS_d^{m+1}\res\Sigma_0$ is equivalent, up to geometric constants, to the measure $|\tau_{\Sigma_0,\cV}|\,\cH^m_{|\cdot|}\res\Sigma_0$, hence in particular $\cS^{m+1}_d(\Sigma_1)>0$, therefore
\[
\cH^{m+1}_d(\Sigma)\geq \cH^{m+1}_d(\Sigma_1)>0\,,
\]
so the proof is complete. 
\end{proof}

\subsection{Formal horizontality of some BV graphs}
Our previous arguments also allow us to establish a kind of ``generalized horizontal tangency'' of BV functions
whose graph satisfies the metric constraint \eqref{hormet2}, as explained in the introduction.
In fact, by the arguments in the proof of Theorem~\ref{H_dm},
it is not difficult to establish the following result.
\begin{Teo}\label{BVg}
Let $2\leq\alpha<3$ and let $g:(0,1)^2\to\R$ be a BV function such that its graph 
\[G=\{(x_1,x_2,g(x)):\; 0<x_1,x_2<1\} \quad\mbox{satisfies}\quad \cH^\alpha_d(G)<+\infty\,,
\]
where $d$ is the SR-distance of $\H^1$, identified with $\R^3$ by the coordinates associated to the
vector fields of \eqref{heisv}. Then the approximate gradient $\mbox{\rm ap}\:\nabla g$ 
almost everywhere satisfies 
\begin{equation}\label{gx}
\mbox{\rm ap}\,\nabla g(x)=(-x_2,x_1)\,.
\end{equation}
\end{Teo}
\begin{Rem}\rm 
As already mentioned in the introduction, the existence of BV functions that satisfy the assumptions 
of Theorem~\ref{BVg} with $\alpha=2$ has been proved by Z. M. Balogh, R. Hoefer-Isenegger and J. T. Tyson, 
\cite{BHIT}.
The existence of BV functions whose absolutely continuous part of the distributional gradient almost everywhere 
equals a vector field with nonvanishing curl is a special instance of a general result due to G. Alberti, \cite{Alberti91}.
\end{Rem}

\end{document}